\documentclass{icm2010}
\usepackage[T1]{fontenc}
\usepackage[latin9]{inputenc}
\usepackage{amsmath}
\usepackage{graphicx}
\usepackage{amssymb}

\makeatletter
\numberwithin{equation}{subsection}
\theoremstyle{plain}
\newtheorem{thm}[subsection]{Theorem}
\theoremstyle{definition}
\newtheorem{example}[subsection]{Example}
\theoremstyle{definition}
\newtheorem{defn}[subsection]{Definition}
\theoremstyle{plain}
\newtheorem{lem}[subsection]{Lemma}
\theoremstyle{plain}
\newtheorem{prop}[subsection]{Proposition}
\theoremstyle{plain}
\newtheorem{cor}[subsection]{Corollary}
 \ifx\proof\undefined\
   \newenvironment{proof}[1][\proofname]{\par
     \normalfont\topsep6\p@\@plus6\p@\relax
     \trivlist
     \itemindent\parindent
     \item[\hskip\labelsep
           \scshape
       #1]\ignorespaces
   }{%
     \endtrivlist\@endpefalse
   }
   \providecommand{\proofname}{Proof}
 \fi

\title[Subfactors and Random Matrices]{Free probability, Planar algebras, Subfactors and Random Matrices.} 
\author[D. Shlyakhtenko]{Dimitri Shlyakhtenko\thanks{Research supported by NSF grants DMS0555680 and DMS0900776}} 
\contact[shlyakht@math.ucla.edu]{Department of Mathematics, UCLA, Los Angeles, CA 90095, USA}

\makeatother

\begin{document}

\begin{abstract}
To a planar algebra $\mathcal{P}$ in the sense of Jones we associate
a natural non-commutative ring, which can be viewed as the ring of
non-commutative polynomials in several indeterminates, invariant under
a symmetry encoded by $\mathcal{P}$. We show that this ring carries
a natural structure of a non-commutative probability space. Non-commutative
laws on this space turn out to describe random matrix ensembles possessing
special symmetries. As application, we give a canonical construction
of a subfactor and its symmetric enveloping algebra associated to
a given planar algebra $\mathcal{P}$. This talk is based on joint
work with A. Guionnet and V. Jones.
\end{abstract}


\maketitle

\section{Introduction.}

The aim of this paper is to explore the appearance of planar algebra
structure in three areas of mathematics: subfactor theory; free probability
theory; and random matrices. 

Jones' subfactor theory has lead to a revolution in understanding
what may be termed {}``quantum symmetry''. The standard invariant
of a subfactor --- the so-called lattice of higher relative commutants,
or {}``$\lambda$-lattice'' \cite{popa:standardlattice,gdhj} is
a remarkable mathematical object, which can represent a very general
type of symmetry. For example, a subfactor inclusion (and so its standard
invariant) can be associated to a Lie group representation. In this
case, the vector spaces that make up the standard invariant are the
spaces of intertwiners between tensor powers of that representation.
Thus the standard invariant of such a subfactor can be used to encode
the representation theory of a Lie group, and thus symmetries associated
with Lie group actions.

In his groundbreaking paper \cite{jones:planar,jones:graphPlanar}
Jones (building on an earlier algebraic axiomatization of standard
invariants by Popa \cite{popa:standardlattice}) showed that there
is a striking way to characterize standard invariants of subfactors:
these are exactly \emph{planar algebras }(see \S\ref{sub:Planar-algebras}
below for a definition). Very roughly, one can think of a planar algebra
as a sequence of vector spaces consisting of vectors invariant under
some {}``quantum symmetry'', together with very general ways (dictated
by planar diagrams) of producing new invariant vectors from existing
vectors. The planar algebra thus \emph{encodes} the underlying symmetry.
In the context of the present paper, we shall use the terms {}``quantum
symmetry'' and {}``planar algebra'' interchangeably. 

Curiously, planar diagrams also occur in random matrix theory. Certain
random multi-matrix ensembles (see \ref{sub:Free-analogs-ofGibbs}
below) are asymptotically described by combinatorics involving counting
\emph{planar maps} (these objects are very much like planar diagrams
appearing in the definition of planar algebras). This fact has been
discovered and extensively used by physicists, starting from the works
of 't Hooft, Brezin, Iszykson, Parisi, Zuber and others (see e.g.
\cite{thooft:planardiagram,brezin-Izykson-Parisi-Zuber:Planar}).
A rigorous proof of convergence was obtained by Guionnet and Maurel
Segala (see \cite{guionnet:icm,guionnet-maurelSegala:ALEA} and references
therein) and Ercolani and McLaughlin \cite{erclani-mclaughlin:enumerationMaps}. 

Finally, turning to Voiculescu's free probability theory \cite{DVV:book},
it was shown by Speicher \cite{cumulants} and others that many important
free probability laws (such as the semicircle law, the free Poisson
law and so on) have combinatorial descriptions involving counting
planar objects (such as non-crossing partitions, which are also very
closely related to planar diagrams).

Thus one is faced with two natural questions. First, why do these
planar structures appear in these three areas? And second, how can
these similarities be exploited?

Concerning the first question, we do not know a fully satisfactory
answer. However, if one grants that planar structure is necessary
to describe {}``quantum symmetries'' (i.e., subfactors), then one
is able to find explanations for appearances of planar structure in
free probability theory and random matrices. We show that one has
a natural notion of a \emph{non-commutative probability law having
a quantum symmetry} --- this law is given by a trace on a ring naturally
associated to a planar algebra. Mathematically, this is accomplished
by a {}``change of rings'' procedure, where we replace the ring
of non-commutative polynomials in $K$ variables with a certain canonical
ring associated to a given planar algebra (see \S\ref{sub:planarAlgebraProbSpace}).
This {}``change of rings'' is analogous to the passage from some
probability space $\Omega$ to the quotient space $\Omega/G$ in the
case that the laws of some family of random variables are invariant
under the action of a group $G$.

Also, we show how to construct random matrix ensembles, which asymptotically
give rise to a non-commutative law with a given quantum symmetry. 

This means that any time one considers a natural equation in free
probability theory, or a natural equation giving the asymptotics of
a random matrix ensemble, this equation must make sense not only as
an equation involving polynomials in $K$ non-commuting indeterminates,
but also arbitrary planar algebra elements. Thus the equation (and
so its solutions) must have a natural planar structure.

Concerning the second question, we give a number of applications of
our techniques. One such application is a version of the ground-breaking
theorem of Popa \cite{popa:standardlattice,shlyakht-popa:universal}
which states that every planar algebra $\mathcal{P}$ arises from
a subfactor $N\subset M$ with $N,M$ isomorphic to free group factors.
It turns out that both $N$ and $M$ can in fact be chosen to be natural
non-commutative probability spaces {}``in the presence of the symmetry
$\mathcal{P}$''. On the random matrix side, our approach gives a
mathematical framework to formulate the work of a number of physics
authors \cite{eynard-zinnJustin:OnModel,kostov:OnModel,zinn-justin:quantumOnModel}
on the so-called $O(n)$ matrix model. In fact, using our techniques
one can make rigorous sense of the $O(n)$ matrix model for $n\in\{2\cos\frac{\pi}{n}:n\geq3\}\cup[2,+\infty)$
(non-integer values of $n$ are used in the physics literature).

The remainder of the paper is organized as follows. We first discuss
some basic notions from free probability theory and subfactors. Next,
we discuss a notion of a non-commutative probability law having a
symmetry encoded by a planar algebra $\mathcal{P}$ and present some
applications to subfactor theory. Finally, we show that one can construct
random matrix ensembles that model certain non-commutative laws with
a given planar algebra symmetry $\mathcal{P}$, and explain connections
with a class of random matrix ensembles used in the physics literature,
and derive some random matrix consequences. 

This paper is based on the joint work with A. Guionnet and V.F.R.
Jones \cite{guionnet-jones-shlyakhtenko1,guionnet-jones-shlyakhtenko2}.

\section{Background and basic notions: Free probability and non-commutative
probability spaces.}

\subsection{Non-commutative probability spaces}

Recall (see for example \cite{DVV:book}) that an algebraic non-commutative
probability space $(A,1_{A},\tau)$ consists of an algebra $A$ with
unit $1_{A}$ and a unital linear functional $\tau:A\to\mathbb{C}$.
We often make the assumption that $A$ is a $*$-algebra and $\tau$
is a \emph{trace}, i.e., $\tau(ab)=\tau(ba)$ for all $a,b\in A$.
Elements of $A$ are called \emph{non-commutative random variables}.
Here are a few examples:
\begin{example}
(a) If $(\mathfrak{X},\mu)$ is a measure space and $\mu$ is a probability
measure, then $(A=L^{\infty}(\mathfrak{X},\mu),1_{A},f\stackrel{\tau}{\to}\int fd\mu)$
is a non-commutative probability space. \\
(b) For any $N$, the algebra of $N\times N$ matrices $(A=M_{N\times N}(\mathbb{C}),1_{A}=\textrm{Id},\tau=\frac{1}{N}Tr)$
is a non-commutative probability space.\\
(c) Consider $A=M_{N\times N}(L^{\infty,-}(\mathfrak{X},\mu))$,
with $(\mathfrak{X},\mu)$ as in (a). Thus elements of $A$ are \emph{random
matrices}. Then $(A,1_{A},\mathbb{E}(\frac{1}{N}Tr(\cdot)))$ is a
non-commutative probability space.
\end{example}
Note that in all of these examples, $\tau$ is a trace: $\tau(xy)=\tau(yx)$. 

In order to be able to do analysis on non-commutative probability
spaces we make the assumption that the algebra $(A,1_{A},\tau)$ is
represented (by bounded or unbounded operators) on a Hilbert space
$H$ by a faithful unital representation $\pi$, so that $\tau(a)=\langle\Omega,\pi(a)\Omega\rangle$
for some fixed vector $\Omega\in H$.

Elements of non-commutative probability spaces are called non-commutative
random variables.

\subsection{Non-commutative laws}

Given $K=1,2,\dots$ classical real random variables $X_{1},\dots,X_{K}$,
which we can think of as an $\mathbb{R}^{K}$-valued function $X$
on some probability space $(\mathfrak{X},\mu)$, their joint law is
defined to be the push-forward by $\tau=X_{*}\mu$ of $\mu$ to a
probability measure on $\mathbb{R}^{K}$. If $\mu$ has finite moments,
we obtain a linear functional on the algebra of polynomials on $\mathbb{R}^{K}$.

By analogy, given non-commutative random variables $X_{1},\dots,X_{K}\in A$,
their \emph{non-commutative law $\tau_{X_{1},\dots,X_{K}}$} is the
linear function on the algebra of all non-commutative polynomials
in $K$ indeterminates $\mathbb{C}[t_{1},\dots,t_{K}]$ obtained by
composing $\tau$ with the canonical map sending $t_{j}$ to $X_{j}$.
In other words\[
\tau_{X_{1},\dots,X_{K}}(P(t_{1},\dots,t_{K}))=\tau(P(X_{1},\dots,X_{K}))\]
for any non-commutative polynomial $P$. 

If $K=1$, non-commutative laws are the same as commutative laws,
modulo identification of measures with linear functionals they induce
on polynomials by integration. For example, in the case of a single
self-adjoint matrix $Y\in(M_{N\times N},\frac{1}{N}Tr)$, its non-commutative
law corresponds to integration against the measure $\mu_{Y}=\frac{1}{N}\sum\delta_{\lambda_{j}}$,
where $\lambda_{1},\dots,\lambda_{N}$ are the eigenvalues of $Y$.
If $Y$ is a random matrix, its non-commutative law captures the expected
value of the random spectral measures associated to $Y$, $\mathbb{E}(\mu_{Y})$.

The classical notion of independence of random variables can be reformulated
algebraically by stating that $(X_{1},\dots,X_{K})$ is independent
from $(X_{K+1},\dots,X_{K+L})$ in a non-commutative probability
space $(A,\tau)$ if the law of $(X_{1},\dots,X_{K+L})\in(A,\tau)$
is the same as that of the variables\[
(\alpha_{1}(X_{1}),\dots,\alpha_{1}(X_{K}),\alpha_{2}(X_{K+1}),\dots,\alpha_{2}(X_{K+L}))\in(A\otimes A,\tau\otimes\tau).\]
Here $\alpha_{1}(X)=X\otimes1$, $\alpha_{2}(X)=1\otimes X$ are two
natural embeddings of $A$ into $A\otimes A$. 

Voiculescu developed his \emph{free probability theory} (see e.g.
\cite{DVV:book}) around another notion of independence, free independence.
For this notion, we say that $(X_{1},\dots,X_{K})$ is freely independent
from $(X_{K+1},\dots,X_{K+L})$ in a non-commutative probability
space $(A,\tau)$ if the law of $(X_{1},\dots,X_{K+L})\in(A,\tau)$
is the same as that of the variables\[
(\alpha_{1}(X_{1}),\dots,\alpha_{1}(X_{K}),\alpha_{2}(X_{K+1}),\dots,\alpha_{2}(X_{K+L}))\in(A*A,\tau*\tau),\]
where $*$ denotes the free product \cite{DVV:free,DVV:book}, and
$\alpha_{1}$, $\alpha_{2}$ are the natural embeddings of $A$ into
$A*A$ (into the first and second copy, respectively).

If $\tau$ is a non-commutative law satisfying positivity and boundedness
requirements, the GNS construction yields a representation of $\mathbb{C}[t_{1},\dots,t_{K}]$
on $L^{2}(\tau)$ and thus generates a von Neumann algebra $W^{*}(\tau)$.
The non-commutative case here differs significantly from the commutative
case. In the commutative case, $W^{*}(\tau)=L^{\infty}(\mathfrak{X})$,
and, notably, all measure spaces $\mathfrak{X}$ are isomorphic (at
least for laws $\tau$ which are non-atomic). In the non-commutative
case, the von Neumann algebras $W^{*}(\tau)$ are much more diverse,
and it is in general a very difficult and challenging question to
decide, for two laws $\tau,\tau'$, when $W^{*}(\tau)\cong W^{*}(\tau')$,
or to somehow identify the isomorphism class of $W^{*}(\tau)$.

\section{Symmetries: Subfactors, Planar algebras, and non-comm\-utative laws}

\subsection{Non-commutative laws with quantum symmetry}

Consider a complex-valued classical random variable $Z$; thus we
actually have a pair of random variables $Z,\bar{Z}$, whose joint
law is described by a probability measure $\mu$ on $\mathbb{C}=\mathbb{R}^{2}$:
for any function of two variables $f(x,y)$, we are interested in
the value\[
\iint f(z,\bar{z})d\mu(z,\bar{z}).\]
In this way, the law of $(Z,\bar{Z})$ is a functional on the space
of functions on $(-\infty,\infty)\times(-\infty,\infty)$. 

Assume that we know that the law of $(Z,\bar{Z})$ is invariant under
rotations: $(Z,\bar{Z})\sim(wZ,\bar{w}\bar{Z})$ for any $w\in\mathbb{C}$,
$|w|=1$. Then the joint law of $(Z,\bar{Z})$ is completely determined
by its {}``radial part'', the integrals of the form\[
\int g(|z|)d\mu(z,\bar{z}),\]
and thus defines a linear functional on the space of rotation-invariant
functions, i.e., effectively on the space of functions on $[0,+\infty)=\mathbb{C}^{2}/\mbox{rotation}$. 

Thus the \emph{presence of a symmetry dictates that we use a different
probability space}. Our aim is to extend this observation to the non-commutative
setting, allowing the most general notions of symmetry possible.

We defined a non-commutative probability law to be a linear functional
$\tau$ defined on the algebra $A=\mathbb{C}[X_{1},\dots,X_{K}]$
of non-commutative polynomials in $K$ variables. If symmetries are
present, this choice of the algebra $A$ may not be suitable. In this
case the algebra $A$ (the non-commutative analog of the ring of polynomials
on $\mathbb{R}^{K}$) must be replaced by the analog of the ring of
functions on a different algebraic variety. For instance, one may
be interested in $*$-probability spaces, i.e., we want to have an
algebra $A$ that has a non-trivial adjoint operation (involution).
This can be accomplished by considering the algebra $B=\mathbb{C}[X_{1},\dots,X_{K},X_{1}^{*},\dots,X_{K}^{*}]$
and defining $X_{j}^{*}$ to be the adjoint of $X_{j}$. An even more
interesting situation is the case that our algebra $B$ has a natural
symmetry. For example, we may consider the action of the unitary group
$U(K)$ on $B$ given on the generators by\begin{equation}
U\cdot X_{k}=\sum U_{ik}X_{i},\qquad U\cdot X_{k}^{*}=\sum\overline{U_{ik}}X_{i}^{*},\qquad\quad U=(U_{ij}).\label{eq:UkAction}\end{equation}
In this case we may only be interested in a part of $B$, the algebra
$B^{U(K)}$ consisting of $U(K)$-invariant elements. One can easily
see that $B^{U(K)}$ is not even a finitely-generated algebra, but
it is the natural non-commutative probability space on which to define
$U(K)$-invariant laws.

More generally, in this paper we will be interested in non-commutative
laws defined on a class of {}``symmetry algebras'', which are the
analogs of algebras such as $B^{U(K)}$ above for more general symmetries
(including actions of quantum groups). 

As is well-known, subfactor theory of Jones provides a framework for
considering such very general symmetries. To formalize our notion
of a {}``non-commutative probability law with a quantum symmetry'',
we shall first review Jones' notion of planar algebras \cite{jones:planar,jones:graphPlanar}.

\subsection{The standard invariant of a subfactor: spaces of intertwiners}

Planar algebras \cite{jones:planar,jones:graphPlanar} were introduced
by Jones in his study of invariants of subfactors of II$_{1}$ factors.

Let $M_{0}\subset M_{1}$ be an inclusion of II$_{1}$ factors of
finite Jones' index \cite{jones:index,gdhj}. Then $M_{1}$ can be
regarded as a bimodule over $M_{0}$ by using the left and right multiplication
action of $M_{0}$ on $M_{1}$. Using the operation of the relative
tensor product of bimodules (see e.g. \cite{connes:correspondences,popa:correspondences,connes:ncgeom,bisch:bimodules})
one can construct other $M_{0},M_{0}$-bimodules by considering tensor
powers\[
M_{k}=\underbrace{M_{1}\otimes_{M_{0}}\otimes\cdots\otimes_{M_{0}}M_{1}}_{k}.\]
One can then consider the intertwiner spaces\[
A_{ij}=\operatorname{Hom}_{M_{0},M_{0}}(M_{i},M_{j})\]
consisting of all homomorphisms from $M_{i}$ to $M_{j}$, which are
linear for both the left and the right action of $M_{0}$. Because
the index of $M_{0}\subset M_{1}$ is finite, these spaces turn out
to be finite-dimensional. The system of intertwiner spaces $A_{ij}$
has more structure than the algebra structure of the individual $A_{ij}$'s.
For example, having an intertwiner $T:M_{i}\to M_{j}$ one can also
construct an {}``induced representation'' intertwiner $T\otimes1:M_{i+1}\to M_{j+1}$.
More generally, one can restrict intertwiners, take their tensor products,
etc., thus providing many operations involving elements of the various
$A_{ij}$'s.

The following example explains how classical representation theory
of a Lie group can be viewed in subfactor terms. Similar examples
exist also in the case of quantum group representations:
\begin{example}
Let $G$ be a Lie group and $V$ be an irreducible representation
of $G$, and denote by $V^{op}$ the representation on the dual of
$V$. Let $M$ be a II$_{1}$ factor carrying an action of $G$ satisfying
a technical condition of being properly outer (such an action always
exists with $M$ a hyperfinite II$_{1}$ factor or a free group factor).
Consider the {}``Wassermann-type'' inclusion\[
M_{0}=M^{G}\subset(M\otimes End(V))^{G}=M_{1}.\]
Here $N^{G}$ denotes the fixed points algebra for an action of $G$
on $N$, and $G$ acts on $End(V)$ by conjugation. Then\begin{alignat*}{1}
\operatorname{Hom}_{M_{0},M_{0}}(M_{k}) & =\operatorname{Hom}_{G}(\underbrace{V\otimes V^{op}\otimes\cdots\otimes V\otimes V^{op}}_{k})\end{alignat*}
is the space of all $G$-invariant linear maps on $(V\otimes V^{op})^{\otimes k}$.
\end{example}
The main theorem of Jones \cite{jones:planar,jones:graphPlanar} is
that there is a beautiful abstract characterization of systems of
intertwiner spaces associated to a subfactor (also called {}``standard
invariants'', {}``$\lambda$-lattices'', systems of higher-relative
commutants): such systems are exactly the \emph{planar algebras}.
His proof relied on an earlier axiomatization of $\lambda$-lattices
by Popa \cite{popa:standardlattice}.

\subsection{Planar algebras\label{sub:Planar-algebras}}

To state the definition of a planar algebra, let us introduce the
notion of a \emph{planar tangle $T$ }with $r$ input disks or sizes
$k_{1},\dots,k_{r}$ and output disk of size $k$ (we'll write $\mathcal{T}(k_{1},\dots,k_{r};k)$
for the set of such tangles). Such a tangle is given by drawing (up
to isotopy on the plane) $r$ {}``input'' disks $(D_{j}:j=1,\dots,r)$
inside the {}``output'' disk $D$. Each disk $D_{l}$ has $2k_{l}$
points marked on its boundary (one of which is marked as the {}``first''
point). The output disk $D$ has $2k$ points marked on its boundary,
one of which is marked {}``first''. Furthermore, all marked boundary
points are connected to other marked points by non-crossing paths.%
\footnote{One also assumes that the connected components of $D\setminus\bigcup_{j}D_{j}$
are colored by two colors, so that adjacent regions are colored by
different colors. We shall, however, ignore this part of this structure
in this paper.%
}

\begin{figure}
\begin{centering}
\includegraphics[bb=0bp 0bp 230mm 6.2999999999999998cm,clip,scale=0.543]{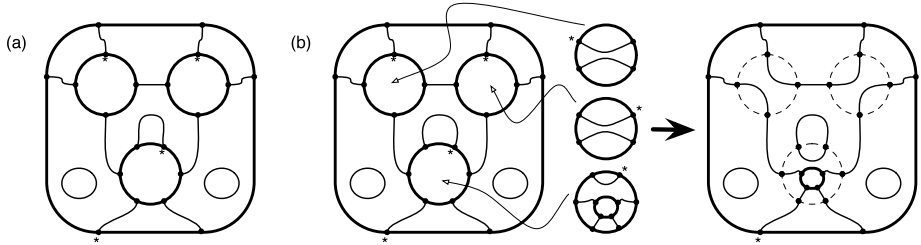}
\par\end{centering}

\caption{\label{fig:Planar-tangles}Planar tangles; composing planar tangles.}

\end{figure}
Figure \ref{fig:Planar-tangles}(a) shows an example of a planar tangle
in $\mathcal{T}(3,3,2;3)$; the first point on each interior disk
is labeled by a $*$. Note that tangles may contain loops which are
not connected to any interior disks.

Tangles can be composed by gluing the output disk of one tangle inside
an input disk of another tangle in a way that aligns points marked
{}``first'' and preserves the orientation of boundaries (see Figure
\ref{fig:Planar-tangles}(b), which illustrates the composition of
a tangle in $\mathcal{T}(3,3,2;3)$ with three tangles, from $\mathcal{T}(4;6)$,
$\mathcal{T}(;4)$ and $\mathcal{T}(;6)$). (This is only possible
if disks are of matching sizes). 
\begin{defn}
Let $(P_{k}:k=0,1,2,\dots)$ be a collection of vector spaces. We
say that $(P_{k})_{k\geq0}$ forms a planar algebra if any planar
tangle\emph{ }$T\in\mathcal{T}(k_{1},\dots,k_{r};k)$ gives rise to
a multi-linear operation $Op(T):P_{k_{1}}\otimes\cdots\otimes P_{k_{r}}\to P_{k}$
in such a way that the assignment $T\to Op(T)$ is natural with respect
to composition of tangles and of multilinear maps. 

Very roughly, one should think of the spaces $P_{k}$ as the space
of {}``intertwiners'' of degree $2k$ for some quantum symmetry
(see \S\ref{sub:Planar-algebras-of-polynomials} below). The various
operations $Op(T)$ correspond to the various ways of combining such
intertwiners to form new intertwiners.

We also often make the assumption that the space $P_{0}$ is one-dimensional
and all $P_{k}$ are finite-dimensional. In particular, a tangle $T$
with no input disks and one output disk with zero marked points and
no paths inside gives rise to a basis element of $P_{0}$, which we'll
denote by $\emptyset$. If we instead consider a tangle $T'$ with
no input disks, one output disk with no marked points, and a simple
closed loop inside of the output disk, then $T'$ produces an element
$\delta\emptyset$ in $P_{0}$ (where $\delta$ is some fixed number).
Furthermore, it follows from naturality of composition of tangles
that if some tangle $T$ is obtained from a tangle $T'$ by removing
a closed loop, then $Op(T)=\delta Op(T')$. 

The tangle in Figure \ref{fig:Canonical-bilinear-form-and-TL}(a)
gives rise to a bilinear form on each $A_{k}$, which we assume to
be non-negative definite. We endow each $P_{k}$ with an involution
compatible with the action of orientation-preserving planar maps on
tangles. Finally, we assume a spherical symmetry, so that we consider
tangles up to isotopy on the sphere (and not just the plane).

A planar algebra satisfying these additional requirements is called
a \emph{subfactor planar algebra} with parameter $\delta$. It is
a famous result of Jones \cite{jones:index} that $\delta\in\{2\cos\frac{\pi}{n}:n\geq3\}\cup[2,+\infty)$,
and all of these values can occur.
\end{defn}

\subsection{Examples of planar algebras}

Planar algebras can be thought of as families of linear spaces consisting
of vectors {}``obeying a symmetry'', where the word symmetry is
taken in a very generalized sense (such {}``symmetries'' include
group actions as well as quantum group actions). We consider a few
examples:

\subsubsection{Planar algebras of polynomials\label{sub:Planar-algebras-of-polynomials}}

Let $X_{1},\dots,X_{K},X_{1}^{*},\dots,X_{K}^{*}$ be indeterminates,
and denote by $A$ the algebra spanned by alternating monomials of
the form $X_{i_{1}}X_{j_{1}}^{*}\cdots X_{i_{k}}X_{j_{k}}^{*}$. Let
$P_{k}$ be the linear subspace of $A$ consisting of all elements
that have degree $2k$. We claim that $\mathcal{P}=(P_{k})_{k\geq0}$
is a planar algebra if endowed with the following structure. Given
a monomial $W=X_{i_{1}}X_{j_{1}}^{*}\cdots X_{i_{k}}X_{j_{k}}^{*}\in P_{k}$,
associate to it the labeled disk $D(W)$ whose $2k$ boundary points
are labeled (clockwise, from the {}``first'' point) by the $2k$-tuple
$(i_{1},j_{1},i_{2},j_{2},\dots,i_{k},j_{k})$. Now given a planar
tangle $T\in\mathcal{T}(k_{1},\dots,k_{r};k)$ and monomials $W_{1},\dots,W_{r}$
of appropriate degrees, we define\[
Op(T)(W_{1},\dots,W_{r})=\sum_{W}C_{W}W.\]
Here the sum is over all monomials $W\in A_{k}$ and $C_{W}$ are
integers obtained as follows. Glue the disks $D(W_{j})$ into the
input disks of $T$ and then the output disk of $T$ into $D(W)$.
We obtain a collection of disks, whose marked boundary points are
connected by curves. Then $C_{W}$ is the total number of ways to
assign integers from $\{1,\dots,K\}$ to these curves, so that each
curve has the same label as its endpoints. ($C_{W}=0$ if no such
assignment exists).

In this case, $\mathcal{P}$ is actually a subfactor planar algebra
with parameter $\delta=K$ (the number of ways to assign an integer
from $\{1,\dots,K\}$ to a closed loop). The corresponding subfactor
inclusion is rather trivial: it corresponds to the $K\times K$ matrix
inclusion $M_{0}=M\subset M\otimes M_{K\times K}(\mathbb{C})=M_{1}$,
for any II$_{1}$ factor $M$. 

Consider the action of the unitary group $U(K)$ on each $P_{k}$
defined by \eqref{eq:UkAction}. In other words, we identify $P_{k}$
with the $k$-th tensor power of $\mathbb{C}^{K}\otimes\overline{\mathbb{C}^{K}}=\operatorname{End}(\mathbb{C}^{K})$,
where $\mathbb{C}^{K}$ is the basic representation of $U(K)$. Then
the linear subspaces $P_{k}^{U(K)}$ consisting of vectors fixed by
the $U(K)$ action turn out to form a planar algebra \emph{$\mathcal{P}^{U(K)}$}
(taken with the restriction of the planar algebra structure of $\mathcal{P}$).
The associated subfactor has the form\[
M^{U(K)}\subset(M\otimes\operatorname{End}(\mathbb{C}^{K}))^{U(K)}.\]

\subsubsection{The Temperley-Lieb planar algebra\label{sub:The-Temperley-Lieb-}}

\begin{figure}

\begin{centering}
\includegraphics[bb=0bp 0bp 350bp 85bp,clip,scale=0.55]{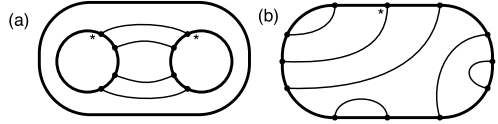}\caption{\label{fig:Canonical-bilinear-form-and-TL}Canonical bilinear form;
Temperley Lieb diagrams.}

\par\end{centering}

\end{figure}

Let $TL_{k}$ be the linear space spanned by tangles $T\in\mathcal{T}(;k)$
with \emph{no} internal disks and $2k$ points on the outer disk.
Such tangles are called \emph{Temperley-Lieb diagrams} (see Figure
\ref{fig:Canonical-bilinear-form-and-TL}(b)). Then $TL=(TL_{k})_{k\geq0}$
is a planar algebra in the following natural way. Given any tangle
$T\in\mathcal{T}(k_{1},\dots,k_{r};k)$ and Temperley-Lieb diagrams
$T_{1},\dots,T_{r}$ , $Op(T)(T_{1},\dots,T_{r})$ is defined to be
the result of gluing the diagrams $T_{1},\dots,T_{r}$ into the input
disks of $T$, provided that we agree that closed loops contribute
a multiplicative factor of $\delta$. $TL$ is actually a subfactor
planar algebra when $\delta$ is in the set of allowed index values
$\{2\cos\frac{\pi}{n}:n\geq3\}\cup[2,+\infty)$.

It should be noted that \emph{any} planar algebra $\mathcal{P}$ contains
a homomorphic image of $TL$; indeed, $TL$ elements arise as $Op(T)$
when $T\in\mathcal{T}(;k)$.

\subsection{Algebras and non-commutative probability spaces arising from planar
algebras}

A planar algebra $\mathcal{P}=(P_{k})_{k\geq0}$ has, by definition,
a large variety of mutli-linear operations. We shall single out the
following bilinear operations $\wedge_{k}$, each of which is an associative
multiplication on $\oplus_{n\geq k}P_{k}$. The operation $\wedge_{k}$
takes $P_{k+n}\times P_{k+m}\to P_{k+m+n}$ and is given by the following
tangle (here $k=2$, $n=1$ and $m=2$):

\begin{center}
\includegraphics[bb=0bp 0bp 250bp 75bp,clip,scale=0.6]{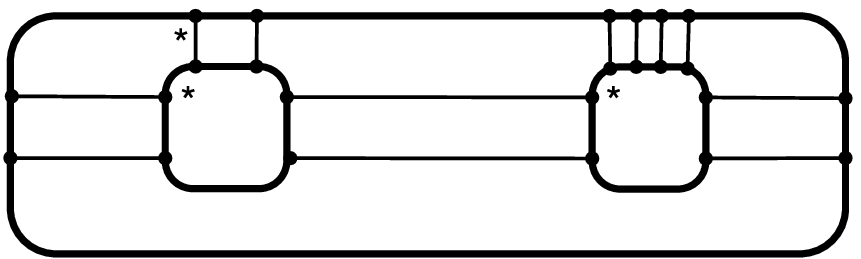}
\par\end{center}

\subsubsection{The product $\wedge_{0}$}

Perhaps the easiest way to see the importance of these operations
is to realize that in the case of planar algebra of polynomials (see
\S\ref{sub:Planar-algebras-of-polynomials}) the multiplication $\wedge_{0}$
is just the ordinary multiplication of polynomials. 

Thus if we think of $\oplus_{k\geq0}P_{k}$ as a linear space consisting
of vectors which are invariant under some {}``quantum symmetry'',
the product $\wedge_{0}$ is a kind of tensor product of these invariants,
and thus $(\mathcal{P},\wedge_{0})$ has the natural interpretation
of the algebra of {}``invariant polynomials''.

\subsubsection{The higher products $\wedge_{k}$}

In the case of the polynomial algebra (\S\ref{sub:Planar-algebras-of-polynomials}),
the product $\wedge_{k}$ corresponds to the product on the algebra
of differential operators of degree $k$. Let us consider such operators
of the form (for simplicity, if $k$ is even)\[
X_{i_{1}}X_{j_{1}}^{*}\cdots X_{i_{k/2}}X_{j_{k/2}}^{*}X_{t_{1}}X_{s_{1}}^{*}\cdots X_{t_{n}}X_{s_{n}}^{*}\partial_{X_{i_{k/2+1}}}\partial_{X_{j_{k/2+1}}^{*}}\cdots\partial_{X_{i_{k}}}\partial_{X_{j_{k}}^{*}}\in P_{k+n}.\]
Such expressions can be multiplied using the convention that $\partial_{X_{s}^{a}}X_{t}^{b}=\delta_{a\neq b}\delta_{s=t}1$,
where $a,b\in\{\ ,*\}$. This is exactly the multiplication $\wedge_{k}$.

Note that the map $\mathcal{E}_{k}$ given by the tangle in Figure
\eqref{fig:VoicTraceAndCup}(c) defines a natural map from $(\mathcal{P},\wedge_{k})$
to $(\mathcal{P},\wedge_{0})$.
\begin{defn}
A \emph{planar algebra law} associated to a planar algebra $\mathcal{P}$
is a linear functional $\tau$ on the algebra \emph{$(\mathcal{P},\wedge_{0})$,
so that $\tau\circ\mathcal{E}_{k}$ is a trace on $(\mathcal{P},\wedge_{0})$
for any $k\geq0$.}
\end{defn}
Since $P_{k}$ can be thought of as the space of vectors with a {}``quantum
symmetry encoded by \emph{$\mathcal{P}$}'', a planar algebra law
is a law having this {}``quantum symmetry''.

\subsection{The Voiculescu trace on $(\mathcal{P},\wedge_{0})$\label{sub:planarAlgebraProbSpace}}

Any planar algebra probability space comes with a natural trace $\tau=\tau_{TL}$
given by the tangle in Figure \eqref{fig:VoicTraceAndCup}(a).

\begin{figure}
\begin{centering}
\includegraphics[bb=0bp 0bp 520bp 160bp,clip,scale=0.55]{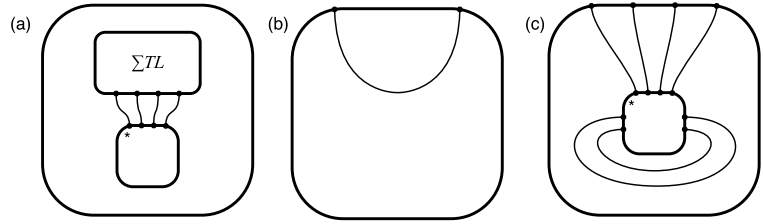}
\par\end{centering}

\caption{\label{fig:VoicTraceAndCup}(a) The Voiculescu trace; here $\sum TL$
stands for the sum of all $TL$ elements with the appropriate number
of strings. (b) The element $\cup$. (c) The map $\mathcal{E}_{k}$
(here $k=2$).}

\end{figure}

\begin{lem}
\cite{guionnet-jones-shlyakhtenko1} \label{lem:cu}(Non-commutative
analog of the $\chi$-squared distribution). Consider the element
$\cup\in TL$ described in Figure \eqref{fig:VoicTraceAndCup}(b).
Then law of $\cup\in TL\subset(\mathcal{P},\wedge_{0},\tau_{TL})$
is the free Poisson law of parameter $\delta$ (see Figure \ref{fig:Free-Poisson-law}).
\end{lem}
\begin{figure}
\centering{}\includegraphics[bb=0bp 150bp 800bp 450bp,clip,width=0.4\columnwidth,height=2cm]{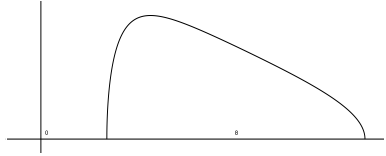}\caption{\label{fig:Free-Poisson-law}Free Poisson law ($\delta=8$).}

\end{figure}

The polynomial planar algebra (see \S\ref{sub:Planar-algebras-of-polynomials})
contains $TL$; one can compute that $\cup=\sum_{i=1}^{K}X_{i}X_{i}^{*}$,
which explains the analogy with the $\chi$-squared law.
\begin{thm}
\cite{guionnet-jones-shlyakhtenko1} Assume that $\mathcal{P}$ is
a subfactor planar algebra. Then trace $\tau_{TL}$ is non-negative
definite. If $\delta>1$, then the von Neumann algebra $M_{0}(\mathcal{P})=W^{*}(\tau_{TL})$
generated in the GNS representation is a II$_{1}$ factor.
\end{thm}
There are several ways in which one can obtain this statement. One
such way is show explicitly that the Hilbert space $L^{2}(\tau_{TL})$
can be identified with the $L^{2}$ direct sum of the spaces making
up the planar algebra \cite{jones-walker-shlyakhtenko:orthogonal}.
To prove that $M_{0}(\mathcal{P})$ is a factor, one first shows that
the element $\cup$ generates a maximal abelian sub-algebra. Thus
the center of $M$ is contained in $W^{*}(\cup)$; some further analysis
shows that the center is in fact trivial.

In a similar way one can prove:
\begin{thm}
\cite{guionnet-jones-shlyakhtenko1} For a subfactor planar algebra\emph{
$\mathcal{P}$, consider} the trace $\tau_{TL}^{n}$ on $(\mathcal{P},\wedge_{n})$
given by $\tau_{TL}\circ\mathcal{E}_{n}$. Then $\tau_{TL}^{n}$ is
non-negative definite, and the von Neumann algebra $M_{n}(\mathcal{P})=W^{*}(\tau_{TL}^{n})$
is a II$_{1}$ factor whenever $\delta>1$. 
\end{thm}

\subsection{Application: constructing a subfactor realizing a given planar
algebra}

The following tangle gives rise to a natural inclusion from $M_{0}(\mathcal{P})$
into $M_{1}(\mathcal{P})$:

\begin{center}
\includegraphics[bb=0bp 0bp 150bp 150bp,clip,scale=0.55]{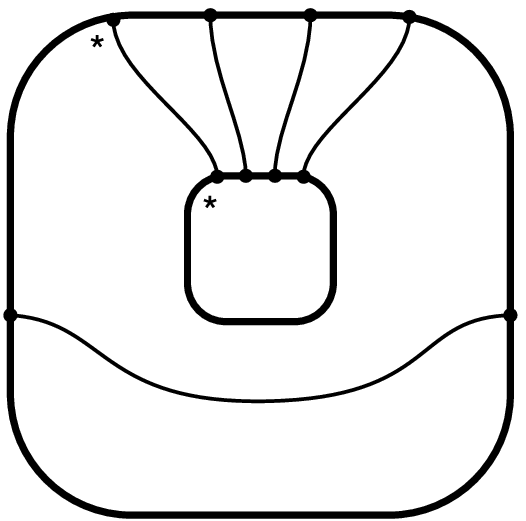}
\par\end{center}

\noindent It turns out that this makes $M_{0}(\mathcal{P})$ into
a finite-index subfactor of $M_{1}(\mathcal{P})$, which \emph{canonically
realizes $\mathcal{P}$:}
\begin{thm}
\cite{guionnet-jones-shlyakhtenko1} (a) The inclusions $M_{0}(\mathcal{P})\subset M_{1}(\mathcal{P})\subset\cdots\subset M_{n-1}(\mathcal{P})\subset M_{n}(\mathcal{P})$
are canonically isomorphic to the tower of basic constructions for
$M_{0}(\mathcal{P})\subset M_{1}(\mathcal{P})$. (b) The planar algebra
associated to the inclusion $M_{0}(\mathcal{P})\subset M_{1}(\mathcal{P})$
is again $\mathcal{P}$.
\end{thm}
In other words, we are able to construct a canonical subfactor realizing
the given planar algebra. A construction that does this was given
earlier by Popa \cite{popa:markov,popa:standardlattice,popa:universalconstructions,shlyakht-popa:universal}
using amalgamated free products. In fact, it turns out that our construction
is related to his; in particular, the algebras $M_{i}(\mathcal{P})$
are isomorphic to certain amalgamated free products \cite{guionnet-jones-shlyakhtenko2,kodiyalam-sunder:depth2,kodiyalam-sunder:interpolated}.
We are able to identify the isomorphism classes of the algebras $M_{j}(\mathcal{P})$:
\begin{thm}
\label{thm:IsomClassTau0}\cite{guionnet-jones-shlyakhtenko2,kodiyalam-sunder:depth2,kodiyalam-sunder:interpolated}
Assume that $\dim P_{0}=\mathbb{C}$, $\delta>1$ and $\mathcal{P}$
is finite-depth of global index $I$. Then\[
M_{0}(\mathcal{P})\cong L(\mathbb{F}_{t})\]
where $t=1+2(\delta-1)I$. More generally, $M_{j}(\mathcal{P})=L(\mathbb{F}_{t_{j}})$
with $t_{j}=1+\delta^{-2j}(\delta-1)I$, $j\geq0$.
\end{thm}
Here $L(\mathbb{F}_{t})$ is the interpolated free group factor \cite{dykema:interpolated,radulescu:subfact}:
$L(\mathbb{F}_{t})=pL(\mathbb{F}_{n})p$ where $p$ is a projection
so that $t-1=\tau(p)^{2}(n-1)$.

Of course, it should be noted that rather than considering von Neumann
algebras $M_{j}(\mathcal{P})=W^{*}(\mathcal{P},\wedge_{j},\tau_{TL}\circ E_{j})$
one can also consider the $C^{*}$-algebras $C^{*}(\mathcal{P},\wedge_{j},\tau_{TL}\circ E_{j})$.
Little is known about their structure.

\subsection{Application: the symmetric enveloping algebra}

\begin{figure}
\begin{centering}
\includegraphics[bb=0bp 0bp 370bp 120bp,clip,scale=0.6]{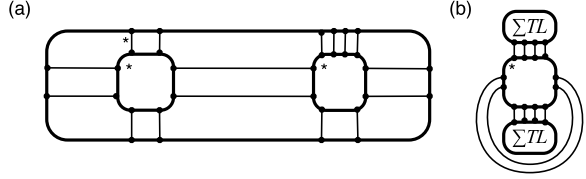}
\par\end{centering}

\caption{\label{fig:envelopingMultandTrace}(a) The multiplication $\boxtimes_{k}$
(there are $k$ horizontal lines joining the input disks). (b) The
trace $\tau\boxtimes_{k}\tau$ (there are $k$ loops).}

\end{figure}

Consider the associative multiplication $\boxtimes_{k}$ defined on
$\oplus_{n\geq k}P_{k}$ by the tangle in Figure \ref{fig:envelopingMultandTrace}(a)
and the trace $\tau\boxtimes_{k}\tau$ on $(\bigoplus_{n\geq k}P_{k},\boxtimes_{k})$
defined in Figure \ref{fig:envelopingMultandTrace}(b).

Let us call $M_{k}\boxtimes M_{k}$ the von Neumann algebra generated
by this algebra in the GNS representation. These algebras are related
to Popa's symmetric enveloping algebra $M_{1}\boxtimes_{e_{0}}M_{1}^{op}$.
For $k=1$ we obtain exactly the symmetric enveloping algebra, at
least in the Temperley-Lieb case.

The symmetric enveloping algebra was introduced by Popa as an important
analytical tool in the study of the {}``quantum symmetry'' behind
a planar algebra. For example, such analytic properties as amenability,
property (T) and so on are encoded by the symmetric enveloping algebra
\cite{popa-symmEnveloping}.

\section{Random matrices and Planar algebras.}

\subsection{GUE and the Voiculescu trace $\tau_{TL}$}

Let $M_{N\times N'}$ denote the linear space of complex $N\times N'$
matrices. Let $K=1,2,\dots$ be an integer, and endow $(M_{N\times N'})^{K}$
with the Gaussian measure\begin{multline*}
d\mu^{(N,N')}(A_{1},\dots,A_{K},A_{1}^{*},\dots,A_{K}^{*})\\
=\frac{1}{Z_{N}}\exp(-\frac{1}{2}NTr(\sum A_{j}^{*}A_{j}))\ dA_{1}\cdots dA_{K}dA_{1}^{*}\cdots dA_{K}^{*}.\end{multline*}
Here $dA_{j}dA_{j}^{*}$ stands for Lebesgue measure on the $j$-th
copy of $M_{N\times N'}$. 

A $K$-tuple of matrices $(A_{1},\dots,A_{K})$ chosen at random from
$(M_{N\times N'})^{K}$ according to this measure is called the Gaussian
Unitary Ensemble (GUE). 

Let $Q$ be a non-commutative polynomial in $X_{1},\dots,X_{K},X_{1}^{*},\dots,X_{K}^{*}$
which is a linear combination of monomials of the form $X_{i_{1}}X_{j_{1}}^{*}\cdots X_{i_{p}}X_{j_{p}}^{*}$
(in other words, we can think of $Q$ as an element of $(\mathcal{P},\wedge_{0})$,
where $\mathcal{P}$ is the planar algebra of polynomials, see \S\ref{sub:Planar-algebras-of-polynomials}).
For each $N,N'$, consider the non-commutative law $\tau^{(N,N')}$
defined by\begin{multline*}
\tau^{(N,N')}(Q)\\
=\int\frac{1}{N}Tr(Q(A_{1},\dots,A_{K},A_{1}^{*},\dots,A_{K}^{*}))d\mu^{(N,N')}(A_{1},\dots,A_{K},A_{1}^{*},\dots,A_{K}^{*}).\end{multline*}
The non-commutative law $\tau^{(N,N')}$ captures certain aspects
of the random multi-matrix ensemble $(A_{1},\dots,A_{K})$. For example,
the value of $\tau^{(N)}\left((A_{1}A_{1}^{*})^{p}\right)$ is the
$p$-th moment of the empirical spectral measure associated to $A_{1}A_{1}^{*}$:
if $\lambda_{1}<\cdots<\lambda_{N}$ are the random eigenvalues of
$A_{1}A_{1}^{*}$, then\[
\tau^{(N)}\left((A_{1}A_{1}^{*})^{p}\right)=\mathbb{E}(\sum\lambda_{j}^{p}).\]

In his seminal paper \cite{DVV:random}, Voiculescu showed that the
laws $\tau^{(N)}$ have a limit as $N\to\infty$; rephrasing slightly
he proved:
\begin{thm}
\label{thm:VoicTracePolynomials}{[}Voiculescu{]} With the above notation,
assume that $N,N'\to\infty$ so that $N'/N\to1$. Then $\tau^{(N)}\to\tau_{TL}$,
where $\tau_{TL}$ is the Voiculescu trace on the planar algebra of
polynomials.
\end{thm}
One can re-derive some well-known random matrix results from this
theorem. For example, combining it with Lemma \ref{lem:cu}, one can
recover convergence of singular values of block random GUE matrices
to the Marcenko-Pastur law \cite{marcenko-pastur}.

\subsection{The case of a general planar algebra}

It turns out that Theorem \ref{thm:VoicTracePolynomials} also holds
in the context of more general planar algebras (i.e., {}``in the
presence of symmetry''). We now describe the appropriate random matrix
ensembles.

\subsubsection{Graph planar algebras}

Our construction relies on the following fact \cite{jones:graphPlanar,guionnet-jones-shlyakhtenko1}:
\begin{prop}
Every planar algebra $\mathcal{P}$ is a subalgebra (in the sense
of planar algebras) of some graph planar algebra $\mathcal{P}^{\Gamma}$.
\end{prop}
Here the graph planar algebra $\mathcal{P}^{\Gamma}$ is a planar
algebra canonically associated to an arbitrary bipartite graph, taken
with its Perron-Frobenius eigenvector $\mu$ (if $\mathcal{P}$ is
finite depth, $\Gamma$ can be taken to be a finite graph). The spaces
$\mathcal{P}_{k}^{\Gamma}$ have as linear bases the sets of closed
paths of length $2k$ on $\Gamma$. The planar algebra structure is
defined in a manner analogous to the case of the polynomial planar
algebra, \S\ref{sub:Planar-algebras-of-polynomials}; see \cite{jones:graphPlanar}
for details. The graph $\Gamma$ can be chosen to be finite if the
planar algebra is finite depth (in particular, if $\delta<2$).

\subsubsection{\label{sub:RandomMatrixGraphGaussian}Random matrix ensembles on
graphs}

Let $\mathcal{P}$ be a planar algebra of finite depth. Thus $\mathcal{P}\subset\mathcal{P}^{\Gamma}$
for some finite bi-partite graph. Let us write $\mu(v)$ for the value
of the Perron-Frobenius eigenvector at a vertex $v$ of $\Gamma$. 

To an oriented edge $e$ of $\Gamma$ which starts at $v$ and ends
at $w$ we associated a matrix $X_{e}$ of size $[N\mu(v)]\times[N\mu(w)]$
(here $[\cdot]$ denotes the integer part of a number). To a path
$e_{1}\cdots e_{n}$ in the graph we associate the product of matrices
$X_{e_{1}}\cdots X_{e_{n}}$ (here $X_{e^{o}}=X_{e}^{*}$ if $e^{o}$
is the edge $e$ but with opposite orientation). 

Thus any element $W\in\bigoplus_{k}P_{k}$ is a specific expression in terms
of the matrices $\{X_{e}\}_{e\in\mathcal{E}(\Gamma)}$. For example,
let $\cup$ be as in Figure \ref{fig:VoicTraceAndCup}(b). Then $\cup=\sum_{e}\sqrt{\frac{\mu(v)}{\mu(w)}}X_{e}X_{e}^{*}$,
the sum taken over all positively oriented edges; here $v$ and $w$
are, respectively, the start and end of $e$. Let us write $W=\sum_{v}W_{v}$,
where $W_{v}$ is in the linear span of closed paths that start at
$v$. Thus for example $\cup_{v}=\sum_{e}\sqrt{\frac{\mu(v)}{\mu(w)}}X_{e}X_{e}^{*}$,
where the sum is taken over all edges $e$ starting at $v$.

With this notation, the expression\[
d\nu_{N}=Z_{N}^{-1}\exp(-N\sum_{v}\mu(v)Tr(\cup_{v}))\prod_{e}dX_{e}\]
makes sense and gives us a probability measure, with respect to which
we can choose our random matrix ensemble $\{X_{e}\}$. 

For any $Q\in P_{k}$, the expression\[
\tau_{N}(Q)=\int\sum_{v}\frac{\mu(v)}{N}Tr(P(Q_{v}(X_{e}:e\in\Gamma)))d\nu_{N}\]
gives rise to a non-commutative law on the non-commutative probability
space $(\mathcal{P}^{\Gamma},\wedge_{0})$ and so in particular on
$(\mathcal{P},\wedge_{0})$. We denote this restriction by $\tau^{(N)}$. 
\begin{thm}
\label{thm:VoicTraceGeneral} With the above notation, $\tau^{(N)}\to\tau_{TL}$,
where $\tau_{TL}$ is the Voiculescu trace on the planar $\mathcal{P}$.
\end{thm}

\subsection{Random matrix ensembles}

More generally, let us assume that we are given a non-commutative polynomial $V(t_{1},\dots,t_{K},t_{1}^{*},\dots,t_{K}^{*})$
which is a sum of monomials of the form $t_{i_{1}}t_{j_{1}}^{*}\cdots t_{i_{p}}t_{j_{p}}^{*}$.  Then 
consider on $(M_{N\times N})^{K}$ the measure\begin{multline}
\label{eq:measuresMuV}
d\mu_{V}^{(N)}(A_{1},\dots,A_{K},A_{1}^{*},\dots,A_{K}^{*})\\
=\frac{1}{Z_{N}}1_{\{\Vert A_{j}\Vert\leq R\}}\exp(-NTr(V(A_{1},\dots,A_{K},A_{1}^{*},\dots,A_{K}^{*})))\\
dA_{1}\cdots dA_{K}dA_{1}^{*}\cdots dA_{K}^{*},\end{multline}
where $dA_{j}$ stands for Lebesgue measure on the $j$-th copy of
$M_{N\times N}$. The constant $Z_{N}$ is chosen so that $\mu_{V}^{(N)}$
is a probability measure (the cutoff $R$ insures that the support
of $\mu_{V}^{(N)}$ is compact). Of course, $R=\infty$ and $V(A_{1},\dots,A_{K})=\sum A_{k}A_{k}^{*}$
corresponds to the Gaussian measure. 

The measures $\mu_{V}^{(N)}$ are matrix analogs of the classical
Gibbs measures $\mu_{V}=Z^{-1}\exp(-V(x))dx$.

Let us call the $K$-tuple of random matrices chosen from $(M_{N\times N}^{sa})^{K}$
at random according to this measure a \emph{random multi-matrix ensemble
}(see \cite[Chapter 5]{guionnet-anderson-zeitouni}\textbf{).}

Certain properties of the random multi-matrix ensemble $A_{1},\dots,A_{K}$
is captured by the non-commutative laws $\tau_{V}^{(N)}$ defined
on the algebra of non-commutative polynomials in $X_{1},\dots,X_{K},X_{1}^{*},\dots,X_{K}^{*}$
by\begin{multline*}
\tau_{V}^{(N)}(Q(X_{1},\dots,X_{K},X_{1}^{*},\dots,X_{K}^{*}))=\\
\int\frac{1}{N}Tr(Q(A_{1},\dots A_{K},A_{1}^{*},\dots,A_{K}^{*}))d\mu_{V}^{(N)}(A_{1},\dots,A_{K},A_{1}^{*},\dots,A_{K}^{*}).\end{multline*}

\subsection{Combinatorial properties of the laws $\tau_{V}^{(N)}$\label{sub:Free-analogs-ofGibbs}}

Remarkably, the laws $\tau_{V}^{(N)}$ have a very nice combinatorial
interpretation. Let $P$, $W_{1},\dots,W_{n}$ be a monomials, and
set $V(t_{1},\dots,t_{K})=(\sum t_{j}t_{j}^{*})+\sum_{j=1}^{n}\beta_{j}W_{j}$.
Define a non-commutative law $\tau_{V}$ by\begin{equation}
\tau_{V}(P)=\sum_{m_{1},\dots,m_{n}\geq0}\sum_{D}\prod_{j=1}^{n}\frac{(-\beta_{j})^{m_{j}}}{m_{j}!}\label{eq:PlanarMaps}\end{equation}
where the summation is taken over all planar tangles $D$ with output
disk labeled by $P$ and having $m_{j}$ interior disks labeled by
$W_{j}$ as in \S\ref{sub:Planar-algebras-of-polynomials}.
\begin{thm}
\cite{guionnet:icm,guionnet-maurelSegala:ALEA} \label{thm:tauVasRMLimit}Let
$P$, $W_{1},\dots,W_{n}$ be monomials, and assume that $V(t_{1},\dots,t_{K})=(\sum t_{j}t_{j}^{*})+\sum_{j=1}^{n}\beta_{j}W_{j}$.
Then for sufficiently small $\beta_{j}$,\[
\tau_{V}^{(N)}(P)=\tau_{V}(P)+O(N^{-2}).\]

\end{thm}
The right-hand side of \eqref{eq:PlanarMaps} would make sense if
we were to replace $P$ and $W_{j}$ by arbitrary elements of an arbitrary
planar algebra (in fact, as written, equation \eqref{eq:PlanarMaps}
can be taken to occur in the planar algebra of polynomials). The term
$\sum t_{j}t_{j}^{*}$ correpsonds to the element $\cup$ defined
in Figure \ref{fig:VoicTraceAndCup}(b). We thus make the following
definition.
\begin{defn}
Let $\mathcal{P}$ be a planar algebra, and assume that $Q\in P_{k},W_{j}\in P_{k_{j}}$,
$j=1,\dots,n$ are elements of algebra $\mathcal{P}$. Let $V_{\beta}=\cup+\sum_{j}\beta_{j}W_{j}$.
We define the associated \emph{free Gibbs law with symmetry $\mathcal{P}$}
to be the planar algebra law\begin{equation}
\tau_{V_{\beta}}(Q)=\sum_{m_{1},\dots,m_{n}\geq0}\sum_{D}\prod_{j=1}^{n}\frac{(-\beta_{j})^{m_{j}}}{m_{j}!}Op(D)(P,\underbrace{W_{1},\dots,W_{1}}_{m_{1}},\dots,\underbrace{W_{n},\dots,W_{n}}_{m_{n}}).\label{eq:tauVforPlanarAlgebra}\end{equation}
Here the summation takes place over all planar tangles $D$ having
one disk of size $k$, $m_{1}$ input disks of size $k_{1}$, $m_{2}$
disks of size $k_{2}$, etc. and no output disks. 
\end{defn}
One can check that in the case of the planar algebra of polynomials,
\eqref{eq:tauVforPlanarAlgebra} is equivalent to \eqref{eq:PlanarMaps}.
\begin{thm}
\label{thm:tauVforPlanarAlgebra}Assume that $Q\in P_{k},W_{j}\in P_{k_{j}}$,
$j=1,\dots,n$ are elements of a finite-depth planar algebra $\mathcal{P}$,
and let $V_{\beta}=\cup+\sum_{j}\beta_{j}W_{j}$. Then for sufficiently
small $\beta$, the free Gibbs law given by \eqref{eq:tauVforPlanarAlgebra}
defines a non-negative trace on $(\oplus_{k\geq0}P_{k},\wedge_{0})$.
\end{thm}
We now show that the laws $\tau_{V_{\beta}}$ arise from random matrix
ensembles, just as in \S\ref{sub:RandomMatrixGraphGaussian} (which
corresponds to $\beta=0$). Once again, we embed $\mathcal{P}$ into
a graph planar algebra $\mathcal{P}^{\Gamma}$ and consider a family
of random matrices $X_{e}$ of size $[N\mu(v)]\times[N\mu(w)]$ labeled
by the edges $e$ of $\Gamma$ (here $[\cdot]$ denotes the integer
part of a number and $\mu$ is the Perron-Frobenius eigenvector of
$\Gamma$). The matrices $X_{e}$ are chosen according to the measure
\[
d\nu_{N}=Z_{N}^{-1}\exp\left(-N\sum_{v}\mu(v)Tr\left((V_{\beta})_{v}\right)\right)\prod_{e}dX_{e}.\]

For any $Q\in P_{k}$, the expression\[
\tau_{N}(Q)=\int\sum_{v}\frac{\mu(v)}{N}Tr(P(Q_{v}(X_{e}:e\in\Gamma)))d\nu_{N}\]
gives rise to a non-commutative law on the non-commutative probability
space $(\mathcal{P}^{\Gamma},\wedge_{0})$ and, by restriction, on
$(\mathcal{P},\wedge_{0})$. We denote this restriction by $\tau_{V_{\beta}}^{(N)}$. 
\begin{thm}
\label{thm:LawsTauV}Assume that $V=\cup+\sum_{j}\beta_{j}W_{j}$
as above. Then there is a $R_{0}>0$ so that for any $R>R_{0}$, there
is a $\beta_{0}>0$ so that for all $|\beta_{j}|<\beta_{0}$, $\tau_{V}^{(N)}\to\tau_{V}$
where $\tau_{V}$ is as in Theorem \ref{thm:tauVforPlanarAlgebra}.
\end{thm}
The finite-depth assumption seems to be technical in nature and is
probably not necessary; it is automatically satisfied if $\delta<2$.

\subsection{Example: $O(n)$ models}

One application of our construction sheds some light on the construction
of so-called $O(n)$ models used by in physics by Zinn-Justin and
Zuber in conjunctions with questions of knot combinatorics \cite{zinn-justin:quantumOnModel,zuber-zinnJustin:countingTanglesLinks}.
For $n$ an integer, the $O(n)$ model is the random matrix ensemble
corresponding to the measure\[
Z_{N}^{-1}\exp(-NTr(V(X_{1},\dots,X_{n})))dX_{1}\cdots dX_{n}dX_{1}^{*}\cdots dX_{n}^{*}\]
where $V$ is a fourth-degree polynomial in $X_{1},\dots,X_{n},X_{1}^{*},\dots,X_{n}^{*}$
, which is invariant under the $U(n)$ action given by \eqref{eq:UkAction}.
In degree $\leq4$, up to cyclic symmetry, the only such invariant
polynomials actually lie in the copy of $TL$ contained in the algebra
$\mathcal{P}^{U(n)}$ in the notation of section \S\ref{sub:Planar-algebras-of-polynomials}:
they are linear combinations of the constant polynomial and the polynomials
$\cup=\sum X_{i}X_{i}^{*}$, $\cup\cup=\sum X_{i}X_{i}^{*}X_{j}X_{j}^{*}$
and $\Cup=\sum X_{i}X_{j}^{*}X_{j}X_{i}^{*}$ (these diagrams are
in $TL\subset\mathcal{P}^{U(n)}$ with parameter $\delta=n$). 

Hence the $O(n)$ model is the random matrix ensemble associated to
the measure\begin{align*}
\mu_{(\beta,n)}^{(N)} & =Z_{N}^{-1}\exp(-NTr(\sum X_{i}X_{i}^{*}+\beta_{i}\sum X_{i}X_{i}^{*}X_{j}X_{j}^{*}+\beta_{2}\sum X_{i}X_{j}^{*}X_{j}X_{i}^{*}).\end{align*}

Thus we are led to consider the laws $\tau_{\beta}$ associated to
the element\[
V_{(\beta,\delta)}=\cup+\beta_{1}\cup^{2}+\beta_{2}\Cup\in TL\]
$\beta=(\beta_{1},\beta_{2})$ for each of the possible parameters
$\delta\in\{2\cos\frac{\pi}{n}:n\geq3\}\cup[2,+\infty)$. From our
discussion we conclude that the limit law associated to the $O(n)$
model is exactly $\tau_{V_{(\beta,\delta=n)}}$.

But since our setting permits non-integer $\delta$, we thus gain
the flexibility of considering the laws $\tau_{V(\beta,\delta)}$
for other values of $\delta$. It can be shown that the values of
$\tau_{V_{(\beta,\delta=n)}}$ on a fixed element of $TL$ are analytic
in $\delta$. Thus the extension we get is exactly the analytic extension
from $n\in\mathbb{Z}$ to $\mathbb{C}$ considered by physicists in
their analysis.

The combinatorics of the resulting law $\tau_{V}$ is governed by
equation \eqref{eq:tauVforPlanarAlgebra}, which is written entirely
in planar algebra terms. In particular, this shows that the $O(n)$
makes rigorous sense for any $\delta\in\{2\cos\frac{\pi}{n}:n\geq3\}\cup[2,+\infty)$
(in the physics literature, the $O(n)$ model was used for non-integer
$n$; the definition involved extending various equations analytically
from $n\in\mathbb{Z}$ to $\mathbb{C}$).

It should be mentioned that $O(n)$ models were introduced in the
physics literature to handle questions of knot enumerations; planar
algebra interpretations of these computations are the subject of on-going
research.

\subsection{Properties of the limit laws $\tau_{V}$}

Because of Theorem \ref{thm:LawsTauV}, fixing a finite-depth planar
algebra $\mathcal{P}$ and a family of elements $V_{\beta}=\cup+\beta W\in\mathcal{P}$,
we obtain a family laws $\tau_{\beta}=\tau_{V_{\beta}}$. These in
turn give rise to a family of von Neumann algebras $W^{*}(\tau_{\beta})$
generated in the GNS representation associated to $\tau_{\beta}$.
When $\beta=0$ these are free group factors (see Theorem \ref{thm:IsomClassTau0}).
Voiculescu conjectured that this is also the case for $\beta\neq0$
sufficiently small.

Using ideas from free probability theory, there has been significant
progress on identifying properties of the associated Neumann algebras
and $C^{*}$-algebras. The key is the following approximation result,
whose proof relies on the theory of free stochastic differential equations
\cite{biane-speicher:stochcalc}. 
\begin{prop}
\cite{guionnet-shlyakhtenko-convex} \label{pro:canApproximate}Assume
that $\mathcal{P}$ is a the planar algebra of polynomials in $K$
variables. Let $S_{1},S_{2},\dots$ be an infinite free semicircular
family generating the $C^{*}$ algebra $B$ with semicircular law
$\tau$, and let $A_{\beta}=C^{*}(\tau_{\beta})$ in the GNS representation
associated to $\tau_{\beta}$. Let $X_{1},\dots,X_{r}\in A_{\beta}$.
Then there is a $\beta_{0}>0$ so that for all $|\beta|<\beta_{0}$
and any $\epsilon>0$ there exists an embedding $\alpha:A_{\beta}\to(A_{\beta},\tau_{\beta})*(B,\tau)$
and elements $Y_{1},\dots,Y_{r}\in B$ so that $\Vert\alpha(X_{j})-Y_{j}\Vert<\epsilon$.
\end{prop}
Using this Proposition, many of the properties of the algebras $A_{\beta}$
can be deduced from those of the algebra $B$. 
\begin{thm}
\cite{guionnet-shlyakhtenko-convex} \label{thm:AlgebrasOfTauV}Let
$V=\cup+\beta W$ be an element of a finite-depth planar algebra $\mathcal{P}$.
Let $\tau_{\beta}$ be the associated law on $(\mathcal{P},\wedge_{0})$.
The von Neumann algebra $M=W^{*}(\tau_{\beta})$ and the $C^{*}$-algebra
$A=C^{*}(\tau_{\beta})$ satisfy:
\begin{enumerate}
\item $M$ is a non-$\Gamma$ II$_{1}$ factor and has the Haagerup property;
\item $A$ is exact;
\item $M$ has Ozawa's property AO and is therefore solid \cite{asher:AO}.
\end{enumerate}
\end{thm}
In the case that $V$ is a polynomial potential (i.e., we are in the
setting of Theorem \ref{thm:tauVasRMLimit}), one can use the results
of \cite{noproj} to prove that $K_{0}(A)=0$ and that $A_{\beta}$
is projectionless. Indeed, if $p\in A$ were a non-trivial idempotent,
then because of Proposition \ref{pro:canApproximate}, $C^{*}(S_{1},S_{2},\dots)\subset C_{\textrm{red}}^{*}(\mathbb{F}_{2})$
would be forced to contain a non-trivial idempotent as well. This
statement has random matrix consequences:
\begin{cor}
\cite{guionnet-shlyakhtenko-convex} Let $\mathcal{P}$ be the planar
algebra of polynomials in $K$ variables, $V=V_{\beta}=\cup+\beta W\in\mathcal{P}$,
and let $\tau_{\beta}=\tau_{V_{\beta}}$ be as in Theorem \ref{thm:tauVasRMLimit}.
Let $Q=Q^{*}\in\mathcal{P}$ be arbitrary polynomial. Let $Q^{(N)}=Q(X_{1},\dots,X_{K})$
be the random matrix obtained by evaluating $Q$ in the random matrices
$(X_{1},\dots,X_{K})$ chosen according to the measure \eqref{eq:measuresMuV}.
Let $\mu^{(N)}$ be the expected value of the spectral measure of
$Q$. Then $\mu^{(N)}\to\mu$ where $\mu$ is a measure with connected
support.\end{cor}
\begin{proof}
Let $Q^{(\infty)}$ denote the element of $C^{*}(\tau_{\beta})$ that
corresponds to the polynomial $Q$ in the GNS construction associated
to $\tau_{\beta}$. Then the law of $Q$ is exactly $\mu$. If the
support of $\mu$ is not connected, the spectrum of $Q\in C^{*}(\tau_{\beta})$
is disconnected. But that means that $C^{*}(\tau_{\beta})$ contains
a non-trivial projection, contradicting Theorem \ref{thm:AlgebrasOfTauV}.
\end{proof}
It turns out that in the presence of symmetry (for non-integer $\delta$)
the algebra $A_{\beta}$ may contain non-trivial projections (even
at $\beta=0$). This phenomenon is not well-understood at this point,
however. It would be interesting to compute the $K$-theory of the
algebras $A_{\beta}$ for general planar algebras $\mathcal{P}$.

\bibliographystyle{amsalpha}

\providecommand{\bysame}{\leavevmode\hbox to3em{\hrulefill}\thinspace}
\providecommand{\MR}{\relax\ifhmode\unskip\space\fi MR }
\providecommand{\MRhref}[2]{%
  \href{http://www.ams.org/mathscinet-getitem?mr=#1}{#2}
}
\providecommand{\href}[2]{#2}

\end{document}